\documentclass[a4paper,11pt,fleqn]{amsart}

\usepackage{amssymb,amsmath,amsthm}
\usepackage{graphicx}
\usepackage{color}
\usepackage{thmtools}
\usepackage{thm-restate}
\usepackage[shortlabels]{enumitem}
\usepackage{mathrsfs}
\usepackage[margin=1in]{geometry}
\usepackage[utf8]{inputenc}
\usepackage[T1]{fontenc}

\newtheorem{theorem}{Theorem}[section]

\newtheorem{proposition}[theorem]{Proposition}

\newtheorem{question}[theorem]{Question}
\newtheorem{corollary}[theorem]{Corollary}

\theoremstyle{definition}

\theoremstyle{remark}
\newtheorem{remark}[theorem]{Remark}

\numberwithin{equation}{section}

\newcommand{\frakb}{\mathfrak{b}}
\newcommand{\frakc}{\mathfrak{c}}

\newcommand{\frakx}{\mathfrak{x}}
\newcommand{\fraky}{\mathfrak{y}}

\newcommand{\eps}{\varepsilon}

\newcommand{\N}{\mathbb{N}}
\newcommand{\R}{\mathbb{R}}

\newcommand{\bB}{\mathcal{B}}

\newcommand{\mM}{\mathcal{M}}
\newcommand{\nN}{\mathcal{N}}

\newcommand{\xX}{\mathcal{X}}

\DeclareMathOperator{\cov}{cov}

\DeclareMathOperator{\non}{non}

\newcommand{\sm}{\setminus}
\newcommand{\sub}{\subseteq}

\DeclareMathOperator{\spn}{span}

\begin{document}

\title{A small remark on small-dimensional normed barrelled spaces}


\author[D.\ Sobota]{Damian Sobota}
\address{Kurt G\"{o}del Research Center, Department of Mathematics, University of Vienna, Vienna, Austria.}
\email{ein.damian.sobota@gmail.com}
\urladdr{www.logic.univie.ac.at/~{}dsobota}

\thanks{The author was supported by the Austrian Science Fund (FWF), ESP 108-N}

\begin{abstract}
Combining the methods of Brian and Stuart with the classical Dvoretzky theorem, we show that no infinite-dimensional Banach space contains a barrelled subspace of (algebraic) dimension $<\mbox{cov}(\mathcal{N})$, the covering number of the Lebesgue null ideal $\mathcal{N}$. Consequently, every infinite-dimensional normed barrelled space has dimension $\ge\mbox{cov}(\mathcal{N})$ and so it is consistent with \textsf{ZFC} that no normed barrelled space has dimension equal to the bounding number $\mathfrak{b}$. 
\end{abstract}


\maketitle

Recall that a topological vector space $E$ is \textit{barrelled} if every absolutely convex closed absorbing subset (\textit{barrel}) of $E$ is a neighborhood of $0$. Barrelled spaces constitute one of the most important classes of topological vector spaces because strong forms of such classical theorems as the Uniform Boundedness Principle or the Open Mapping Theorem hold in them, see \cite{Swa92}. The minimal (algebraic) dimension of infinite-dimensional metrizable barrelled spaces was studied by S\'anchez Ruiz and Saxon in \cite{SRS95,SRS96,SRS96-2}, where it was proved that it is equal to the cardinal characteristic of the continuum called the \textit{bounding number} $\frakb$\footnote{See \cite{BJ95} for the definitions and basic facts concerning the cardinal characteristics $\frakb$, $\non(\mM)$, and $\cov(\nN)$, used in this note.}. 
In \cite{SRS95} they also defined the following cardinal number:
\[\fraky=\min\big\{\kappa\ge\aleph_0\colon\ \text{there is a normed barrelled space of dimension }\kappa\big\},\]
and asked whether it is equal to any standard cardinal characteristic of the continuum. Of course, we have $\fraky\ge\frakb$. The author proved in \cite{Sob19} that consistently $\fraky$ may be strictly smaller than the continuum $\frakc$, i.e. the cardinality of the real line $\R$, which implies that in \textsf{ZFC}, the standard system of axioms of set theory, one cannot prove that $\fraky=\frakc$ (see also \cite{SZ17,SZ19}). 

Brian and Stuart \cite{BS24} recently showed that every infinite-dimensional separable Banach space contains a barrelled subspace of dimension $\non(\mM)$, the minimal cardinality of a non-meager subset of $\R$. It follows in particular that $\fraky\le\non(\mM)$. Moreover, they also proved that if the dual space $E^*$ of a Banach space $E$ contains an isomorphic copy of a Banach space $F\in\{c_0\}\cup\{\ell_p\colon 1\le p<\infty\}$, then no barrelled subspace of $E$ has dimension strictly less than $\cov(\nN)$, the minimal cardinality of a family of the Lebesgue null sets whose union is the interval $[0,1]$ (\cite[Theorem 4.11]{BS24}). The question whether the conclusion of the latter statement holds for all Banach spaces $E$ was left open.

\begin{question}[{\cite[Question 4.13]{BS24}}]\label{question}
Does every barrelled subspace of an infinite-dimensional Banach space have dimension $\ge\cov(\nN)$? In particular, $\fraky\ge\cov(\nN)$?
\end{question}

The aim of this note is to answer Question \ref{question} in affirmative and hence to complete Brian and Stuart's work---see Theorem \ref{thm:main} and Corollary \ref{cor:covN}. Moreover, in Corollary \ref{cor:b_covN_y} we get that the inequality $\frakb<\fraky$ is consistent with \textsf{ZFC}, which sheds new light on the question of  S\'anchez Ruiz and Saxon.

\medskip

Let us briefly discuss the argument. The proof of aforementioned \cite[Theorem 4.11]{BS24} is divided into two cases: separately for $F=\ell_1$ and $F\in\{c_0\}\cup\{\ell_p\colon 1<p<\infty\}$. To establish the second case, the authors of \cite{BS24} introduce a property of Banach spaces denoted $(\dagger)$, prove that $c_0$ and $\ell_p$ for $1<p<\infty$ satisfy $(\dagger)$, and finally show that $(\dagger)$ implies the required conclusion concerning the dimension of barrelled subspaces. However, an analysis of their proof demonstrates that what is in fact used to obtain the last result is the following property of Banach spaces $E$, which we call $(\dagger')$ and which is formally weaker than $(\dagger)$ (see Remark \ref{rem:on_proofs}):
\begin{itemize}[$(\dagger')$]
    \item there are a sequence $(e_k^*)_{k\in\N}$ of unit functionals in $E^*$ and a partition of $\N=\bigcup_{m\in\N}I_m$ into finite intervals such that
    $$\sum_{m\in\N}\frac{\big|\big\{i\in I_m\colon\ |e_i^*(e)|\ge1/m\big\}\big|}{|I_m|}<\infty$$
    for every unit vector $e\in E$.
\end{itemize}
(Here, as usual, $\N$ denotes the set of positive integers, i.e. $\N=\{1,2,3,...\}$.)

It turns out that the property $(\dagger')$ is actually satisfied by \textit{every} infinite-dimensional Banach space, including $\ell_1$, which follows from the classical Dvoretzky theorem asserting that the space $\ell_2$ is \textit{finitely representable} in every infinite-dimensional Banach space $E$, that is, for each finite-dimensional subspace $F$ of $\ell_2$ and each $\eps>0$ there are a finite-dimensional subspace $G$ of $E$ and a surjective isomorphism $T\colon F\to G$ with $\|T\|=1$ and $\|T^{-1}\|<1+\eps$ (cf. \cite[Theorem 12.3.13 and Remark 12.1.3.(a)]{AK16}); see Proposition \ref{prop:daggerprim}. Consequently, by using the arguments presented in the proofs of \cite[Theorems 4.3 and 4.9]{BS24}, in Theorem \ref{thm:main} we get that every barrelled subspace of an infinite-dimensional Banach space must have dimension at least $\cov(\nN)$, which basically answers Question \ref{question}.

\medskip

We are ready to present the proofs.

\begin{proposition}\label{prop:daggerprim}
    Every infinite-dimensional Banach space satisfies $(\dagger')$.
\end{proposition}
\begin{proof}
    Let $E$ be an infinite-dimensional Banach space. Let $(\tilde{e}_n)_{n\in\N}$ be the standard basis of $\ell_2$, i.e. for $j\neq n\in\N$ we have $\tilde{e}_n(j)=0$ and $\tilde{e}_n(n)=1$.
    
    For every $m\in\N$ set $E_m=\spn\big\{\tilde{e}_1,\ldots,\tilde{e}_{m^4}\big\}$, so $E_m$ is an $m^4$-dimensional subspace of $\ell_2$, and let $T_m\colon E_m\to F_m$ be an isomorphism onto a subspace $F_m$ in the dual space $E^*$ given by the Dvoretzky theorem with $\|T_m\|=1$ and $\big\|T_m^{-1}\big\|<2$. For each $m\in\N$ and $1\le i\le m^4$, let 
    \[\alpha_{m,i}=\big\|T_m(\tilde{e}_i)\big\|^{-1},\]
    so $\alpha_{m,i}<2$, and
    \[f_{m,i}^*=\alpha_{m,i}T_m(\tilde{e}_i),\]
    hence $\big\|f_{m,i}^*\big\|=1$. Let $\N=\bigcup_{m\in\N}I_m$ be a partition of $\N$ into consecutive intervals such that $|I_m|=m^4$. For each $k\in\N$, let $m_k\in\N$ and $1\le i_k\le m_k^4$ be such that $k=i_k+\sum_{j=1}^{m_k-1}j^4$, and set 
    \[e_k^*=f_{m_k,i_k}^*.\]

    We claim that the sequences $(e_k^*)_{k\in\N}$ and $(I_m)_{m\in\N}$ witness that $E$ satisfies $(\dagger')$. To this end, fix a unit vector $e\in E$. Let $m\in\N$. Put
    \[J_m^+=\big\{k\in I_m\colon e_k^*(e)\ge1/m\big\},\quad J_m^-=\big\{k\in I_m\colon -e_k^*(e)\ge1/m\big\},\quad J_m=J_m^+\cup J_m^-.\]
    We have
    \[\bigg\|\sum_{k\in J_m^{\pm}}e_k^*\bigg\|=\bigg\|\sum_{k\in J_m^{\pm}}f_{m,i_k}^*\bigg\|\le\|T_m\|\cdot\bigg\|\sum_{k\in J_m^{\pm}}\alpha_{m,i_k}\tilde{e}_{i_k}\bigg\|_{\ell_2}=\bigg(\sum_{k\in J_m^{\pm}}\alpha_{m,i_k}^2\bigg)^{1/2}\le2\sqrt{\big|J_m^{\pm}\big|}.\]
    On the other hand,
    \[\bigg\|\sum_{k\in J_m^+}e_k^*\bigg\|\ge\bigg|\sum_{k\in J_m^+}e_k^*(e)\bigg|=\sum_{k\in J_m^+}e_k^*(e)\ge\big|J_m^+\big|/m,\]
    and similarly
    \[\bigg\|\sum_{k\in J_m^-}e_k^*\bigg\|\ge\bigg|\sum_{k\in J_m^-}e_k^*(e)\bigg|=-\sum_{k\in J_m^-}e_k^*(e)\ge\big|J_m^-\big|/m.\]
    Combining the above inequalities, we get
     \[\big|J_m^{\pm}\big|/m\le\bigg\|\sum_{k\in J_m^{\pm}}e_k^*\bigg\|\le2\sqrt{\big|J_m^{\pm}\big|},\]
     hence, rearranging,
     \[\big|J_m^{\pm}\big|\le4m^2,\]
     and finally
     \[|J_m|=\big|J_m^+\big|+\big|J_m^-\big|\le8m^2.\]

     Using the latter inequality for every $m\in\N$, we get
     \[\sum_{m\in\N}\frac{\big|\big\{k\in I_m\colon\ |e_k^*(e)|\ge1/m\big\}\big|}{|I_m|}=\sum_{m\in\N}\frac{|J_m|}{|I_m|}\le\sum_{m\in\N}\frac{8m^2}{m^4}=8\sum_{m\in N}1/m^2<\infty,\]
     which finishes the proof of the proposition.
\end{proof}

\begin{theorem}\label{thm:main}
    Assume that $A$ is a subset of an infinite-dimensional Banach space $E$ such that the space $\spn A$ is barrelled. Then, $|A|\ge\cov(\nN)$.
\end{theorem}
\begin{proof}
    The proof is the same as of \cite[Theorem 4.9]{BS24}, so for the sake of completeness we only sketch it.

    By Proposition \ref{prop:daggerprim} the space $E$ satisfies $(\dagger')$, so let $(e_k^*)_{k\in\N}$ and $(I_m)_{m\in\N}$ be sequences given by $(\dagger')$ for $E$. Endow each set $I_m$ with the discrete topology, and let $P$ be the standard probability product measure on the compact space $\Omega=\prod_{m\in\N}I_m$, i.e.
    \[P\big(\{f\in\Omega\colon f(m)\in F\}\big)=|F|/|I_m|\]
    for every $m\in\N$ and $F\sub I_m$. Note that by \cite[Theorem 17.41]{Kec95} the measure space $(\Omega,\Sigma,P)$, where $\Sigma$ is the Borel $\sigma$-field of $\Omega$, is measure-preserving isomorphic to the measure space $([0,1],\bB,\lambda)$, where $\bB$ is the standard Borel $\sigma$-field of $[0,1]$ and $\lambda$ denotes the Lebesgue measure on $[0,1]$; consequently, if $\xX$ is a collection of $P$-null subsets of $\Omega$ with $|\xX|<\cov(\nN)$, then $\Omega\sm\bigcup\xX\neq\emptyset$.
    
    For every $f\in\Omega$ and $m\in\N$ set
    \[y_{f,m}^*=me_{f(m)}^*,\]
    so that $\big\|y_{f,m}^*\big\|=m$, and hence $\sup_{m\in\N}\big\|y_{f,m}^*\big\|=\infty$.

    We claim that for every $x\in A$ it holds
    \[\tag{$*$}P\Big(\Big\{f\in\Omega\colon \sup_{m\in\N}\big|y_{f,m}^*(x)\big|=\infty\Big\}\Big)=0.\]
    To see this, let $x\in A$ be non-zero. For all $M\in\N$ we have
    \[P\Big(\Big\{f\in\Omega\colon \sup_{m\in\N}\big|y_{f,m}^*(x)\big|=\infty\Big\}\Big)=P\Big(\Big\{f\in\Omega\colon \sup_{m>M}\big|y_{f,m}^*(x)\big|=\infty\Big\}\Big)\le\]
    \[\le P\Big(\Big\{f\in\Omega\colon \big|y_{f,m}^*(x)\big|\ge\|x\|\text{ for some }m>M\Big\}\Big)=\]
    \[=P\Big(\Big\{f\in\Omega\colon \big|e_{f(m)}^*\big(x/\|x\|\big)\big|\ge1/m\text{ for some }m>M\Big\}\Big)\le\]
    \[\le\sum_{m>M}P\Big(\Big\{f\in\Omega\colon \big|e_{f(m)}^*\big(x/\|x\|\big)\big|\ge1/m\Big\}\Big)=\]
    \[=\sum_{m>M}P\Big(\Big\{f\in\Omega\colon f(m)\in\big\{i\in I_m\colon \big|e_{i}^*\big(x/\|x\|\big)\big|\ge1/m\big\}\Big\}\Big)=\]
    \[=\sum_{m>M}\frac{\big|\big\{i\in I_m\colon \big|e_{i}^*\big(x/\|x\|\big)\big|\ge1/m\big\}\big|}{|I_m|},\]
    which, by $(\dagger')$, 
    converges to $0$ as $M\to\infty$.

    The claim implies that $|A|\ge\cov(\nN)$. If not, then by ($*$) 
    we have
    \[\Omega\sm\bigcup_{x\in A}\Big\{f\in\Omega\colon \sup_{m\in\N}\big|y_{f,m}^*(x)\big|=\infty\Big\}\neq\emptyset,\]
    that is, there is $g\in\Omega$ such that $\sup_{m\in\N}\big|y_{g,m}^*(x)\big|<\infty$ for every $x\in A$. However, since the space $\spn A$ is barrelled, it follows from  \cite[Lemma 4.1]{BS24} that there is $x_0\in A$ for which we have $\sup_{m\in\N}\big|y_{g,m}^*(x_0)\big|=\infty$, which is a contradiction.
\end{proof}

\begin{remark}\label{rem:on_proofs}
As mentioned, the proof of Theorem \ref{thm:main} closely follows the argument presented in the proof of \cite[Theorem 4.9]{BS24}, whose statement formally requires the Banach space $E$ to satisfy a property denoted $(\dagger)$, stronger than $(\dagger')$. However, in fact only property $(\dagger')$ is used in the proof of \cite[Theorem 4.9]{BS24}, not $(\dagger$). More precisely, the reasoning to establish \cite[Theorem 4.9]{BS24} goes exactly as follows: the authors first prove in \cite[Lemma 4.8]{BS24} that $(\dagger)$ implies $(\dagger')$, then---at the very beginning of the proof of \cite[Theorem 4.9]{BS24}---they invoke the lemma to assert that $E$ satisfies $(\dagger')$, and, finally, they proceed with the argument exactly as presented in the proof of Theorem \ref{thm:main} to establish equality ($*$) and so to contradict the barrelledness of $\spn A$.
\end{remark}

Since the isometric image of a normed barrelled space into its completion is still barrelled, Theorem \ref{thm:main} immediately yields the following corollary.

\begin{corollary}\label{cor:covN}
    Every infinite-dimensional normed barrelled space has dimension at least $\cov(\nN)$.
\end{corollary}

Combining Corollary \ref{cor:covN} with the aforementioned results from \cite{BS24} and \cite{SRS95}, we get the following estimates for $\fraky$.

\begin{corollary}\label{cor:y_covN}
    $\max\!\big(\frakb,\cov(\nN)\big)\le\fraky\le\non(\mM)$.
\end{corollary}

Both the inequality $\frakb<\cov(\nN)$ as well as the inequality $\cov(\nN)<\frakb$ is consistent with \textsf{ZFC}, see \cite[Sections 7.5--7.6]{BJ95}. Thus, by Corollary \ref{cor:y_covN}, we have the following consistency result.

\begin{corollary}\label{cor:b_covN_y}
    Let $\frakx\in\!\big\{\frakb,\cov(\nN)\big\}$. It is consistent with \textup{\textsf{ZFC}} that no infinite-dimensional Banach space contains a barrelled subspace of dimension $\frakx$, and so that $\frakx<\fraky$.
\end{corollary}

By \cite[Theorem 6]{SRS95}, there is a metrizable barrelled space of dimension $\frakb$. This, together with Corollary \ref{cor:covN}, yields our last result.

\begin{corollary}
    In any model of \textup{\textsf{ZFC}}$+``\frakb<\cov(\nN)"$, there is an infinite-dimensional metrizable barrelled space of dimension strictly less than the dimension of any infinite-dimensional normed barrelled space. 
\end{corollary}

We finish the note by asking the following question (cf. \cite[Question 4.14]{BS24}).

\begin{question}
Does the equality $\fraky=\non(\mM)$ hold in \textup{\textsf{ZFC}}?
\end{question}

\section*{Acknowledgement}

The author would like to thank Professor Anna Pelczar-Barwacz for a valuable discussion concerning the Dvoretzky theorem.

\end{document}